\documentclass[11pt]{article}
\usepackage[dvips]{graphicx}
\usepackage[dvips,letterpaper]{geometry}
\usepackage{color}
\usepackage{amsmath}
\usepackage{amsfonts}
\usepackage{amssymb}
\usepackage{amsthm}
\usepackage{newlfont}
\usepackage{epsfig}
\usepackage{multirow}
\usepackage{setspace}
\usepackage{hyperref}
\usepackage{cite}
\usepackage{wrapfig}
\usepackage{tikz}
\usepackage{subfig}
\usetikzlibrary{shapes,arrows,automata}

\begin{document}

\newtheorem{assumption}{Assumption}[section]
\newtheorem{definition}{Definition}[section]
\newtheorem{lemma}{Lemma}[section]
\newtheorem{proposition}{Proposition}[section]
\newtheorem{theorem}{Theorem}[section]
\newtheorem{corollary}{Corollary}[section]
\newtheorem{remark}{Remark}[section]
\newtheorem{conjecture}{Conjecture}[section]
\newtheorem{example}{Example}[section]

\small

\title{A Linear Programming Approach to Dynamical Equivalence, Linear Conjugacy, and the Deficiency One Theorem \bigskip}
\author{Matthew D. Johnston\\
Department of Mathematics\\
San Jos\'{e} State University\\
One Washington Square\\
San Jos\'{e}, CA 95192\\[0.1in]
\tt{matthew.johnston@sjsu.edu}}
\date{}
\maketitle

\tableofcontents

\begin{abstract}
\small
The well-known Deficiency One Theorem gives structural conditions on a chemical reaction network under which, for any set of parameter values, the steady states of the corresponding mass action system may be easily characterized. It is also known, however, that mass action systems are not uniquely associated with reaction networks and that some representations may satisfy the Deficiency One Theorem while others may not. In this paper we present a mixed-integer linear programming framework capable of determining whether a given mass action system has a dynamically equivalent or linearly conjugate representation which has an underlying network satisfying the Deficiency One Theorem. This extends recent computational work determining linearly conjugate systems which are weakly reversible and have a deficiency of zero.

\end{abstract}


\bigskip

\section{Introduction}
\label{introduction}

A chemical reaction network is given by sets of reactants which interact according to fixed reaction channels to form new sets of reactants. Under suitable kinetic assumptions, such as spatial homogeneity and sufficient molecularity, these networks can be modeled by a system of autonomous polynomial ordinary differential equations known as a mass action system. The study of mass action systems, and the related area of chemical reaction network theory (CRNT), has been studied increasingly in recent years as the interdisciplinary area of systems biology has become more prominent.

A primary focus of CRNT is on the relationship between the topological structure of the network of interactions and the permissible dynamical behaviors of the corresponding reaction systems under a variety of kinetic assumptions. The canonical paper \cite{H-J1}, published in 1972, introduces the notion of complex-balancing in a reaction network and shows that this condition is sufficient to guarantee very strong asymptotic behavior of the corresponding mass action system. The concurrent papers \cite{F1} and \cite{H} introduce a nonnegative network parameter known as the deficiency and develop the now classical \emph{Deficiency Zero Theorem}. This theorem states that a network which is weakly reversible and has a deficiency of zero is complex balanced for all parameter values, and therefore the corresponding mass action systems possess the associated asymptotic behavior. Networks with a higher deficiency, and in particular a deficiency of one, have also been studied. The most well-known of these is the \emph{Deficiency One Theorem}, which gives conditions sufficient for the uniqueness of steady states of higher deficiency mass action systems \cite{Fe2,F2}. Other deficiency-based results characterizing the existence, number, and nature of steady states have also been derived \cite{Fe4,Boros2012,Sh-F,Ji}.

It is furthermore known that two chemical reaction networks can generate the same mass action system and therefore be dynamically equivalent. In such cases, one network may have a desirable network property such as weak reversibility or a low deficiency while another might not. Significant work has consequently been conducted on determining network representations of mass action systems with desirable structural properties. Mixed-integer linear programming (MILP) frameworks are now established for determining dynamically equivalent network structures which are linearly conjugate \cite{J-S4}, weakly reversible \cite{J-S4,Sz-H-T,Sz2015-3,Sz2014}, detailed and complex balanced \cite{Sz-H,J-S5}, reaction dense and reaction sparse \cite{Sz-H-P}, and have a minimal deficiency \cite{J-S6,Sz2015-2}. In particular, the question of whether a mass action system has a network representation satisfying the assumptions of the Deficiency Zero Theorem was answered in \cite{J-S6}. Many of these features are implemented in the computational package CRNreals \cite{S-B-A}.

One question which remains open is whether a given a mass action system has a dynamically equivalent or linearly conjugate representation satisfying the Deficiency One Theorem. To date, there exists no framework by which to check the following technical assumptions of the theorem: (a) that the sum of the deficiencies of each linkage classes is bounded by one and sums to the overall network deficiency; and (b) that each linkage class contains a single terminal strong linkage class. In this paper, we extend the MILP framework introduced in the papers outlined above to include conditions (a) and (b), and therefore determine whether a network has a representation satisfying the assumptions of the Deficiency One Theorem. We present examples of mass action systems and networks which, while not amenable to the Deficiency One Theorem directly, have dynamically equivalent and/or linearly conjugate systems which are amenable to it.

\section{Background}
\label{background}

In this section, we give the necessary terminology, notation, and background results relevant to the results contained in Section \ref{computational}.

\subsection{Chemical Reaction Networks}
\label{crnsection}

The following is the fundamental object of this paper.

\begin{definition}
\label{crn}
A chemical reaction network is a triple of sets $(\mathcal{S},\mathcal{C},\mathcal{R})$ where:
\begin{enumerate}
\item
The \textbf{species set} $\mathcal{S} = \{ X_1, \ldots, X_m \}$ consists of the elementary chemical species capable of undergoing chemical change.
\item
The \textbf{complex set} $\mathcal{C} = \{ C_1, \ldots, C_n \}$ consists of linear combinations of the species of the form
\[C_i = \sum_{j=1}^m y_{ij} X_j\]
where $y_{ij} \geq 0$ are the \textbf{stoichiometric coefficients}. We will let $y_i = (y_{i1},\ldots,y_{im}) \in \mathbb{R}_{\geq 0}^m$ denote the vector of stoichiometric coefficients corresponding to the complex $C_i \in \mathcal{C}$.
\item
The \textbf{reaction set} $\mathcal{R} = \{ R_1, \ldots, R_r \}$ consists of ordered pairs of complexes $(C_i,C_j) \in \mathcal{R}$. Reactions will also be represented as $C_i \to C_j \in \mathcal{R}$.
\end{enumerate}
We will assume that every reaction network satisfies the following: (i) every species is contained in at least one complex; and (ii) every complex is contained in at least one reaction.
\end{definition}
\noindent Note that we do not include the typical assumption that there are no self-reactions in the network (i.e. reactions of the form $C_i \to C_i$). The algorithm presented in Section \ref{computational} will require that we permit complexes which are isolated from every other complex (i.e. they are not connected to any other complex by any reaction). We will accommodate these isolated complexes by associating each such complex to a self-reaction. Such an allowance does not limited our ability to apply the Deficiency One Theorem (see Lemma \ref{noproblem}).

Associated with each chemical reaction network $(\mathcal{S},\mathcal{C},\mathcal{R})$ is a digraph $G(V,E)$ where $V = \mathcal{C}$ and $E = \mathcal{R}$. This digraph is known in the literature as the reaction graph of the network \cite{H-J1}. Many properties from graph theory have been studied in the context of chemical reaction networks which we now briefly introduce.

Two complexes $C_i$ and $C_j$ are said to be connected if there is a sequence of complexes such that $C_i = C_{\mu(1)} \leftrightarrow C_{\mu(2)} \leftrightarrow \cdots \leftrightarrow C_{\mu(l)} = C_j$ where $C_{\mu(k-1)} \leftrightarrow C_{\mu(k)}$ implies either $C_{\mu(k-1)} \leftarrow C_{\mu(k)}$ or $C_{\mu(k-1)} \rightarrow C_{\mu(k)}$. We say there is a path from $C_i$ to $C_j$ if there is a sequence of complexes such that $C_i = C_{\mu(1)} \rightarrow C_{\mu(2)} \rightarrow \cdots \rightarrow C_{\mu(l)} = C_j$. A subset of complexes $L \subseteq \mathcal{C}$ is called a linkage class if it is maximally connected. The set of linkage class of a network will be denoted $\mathcal{L} = \{ L_1, \ldots, L_{\ell} \}$. A subset of complexes $L \subseteq \mathcal{C}$ is called a strong linkage class if is maximally path-connected; that is, for every two complexes $C_i, C_j \in L$, $C_i \not= C_j$, a path from $C_i$ to $C_j$ implies a path from $C_j$ to $C_i$, but $C_k \not\in L$ implies either no path from $C_i$ to $C_k$ or no path from $C_k$ to $C_i$. A strong linkage class is called terminal if there is no reaction from a complex in the strong linkage class to a complex not in the strong linkage class. The set of terminal strong linkage classes will be denoted $\mathcal{T} = \{ T_1, \ldots, T_t \}$. A network is said to be weakly reversible if every linkage class is a strong linkage class.

To every reaction $C_i \to C_j \in \mathcal{R}$ there is an associated reaction vector $y_j - y_i \in \mathbb{R}^m$ which tracks the net gain/loss of each species as a result of a single instance of the reaction. The stoichiometric subspace of the network is given by:
\[S = \mbox{span} \{ (y_j-y_i) \in \mathbb{R}^m \; | \; C_i \to C_j \in \mathcal{R} \}.\]
The dimension of the stoichiometric subspace will be denoted $s = \mbox{dim}(S)$.

\subsection{Mass Action Systems}
\label{massactionsystems}

In order to determine how the species involved in a chemical reaction network evolve over time, it is necessary to make assumptions on the kinetics of the system. It is common to assume that the system is subject to the law of mass action, which states that the rate of a reaction is proportional to the product of the reactant concentrations. For example, a reaction of the form $X_1 + X_2 \to \cdots$ would have the rate $k \cdot x_1 x_2$ where $x_1 = [X_1]$ and $x_2 = [X_2]$ are the concentrations of $X_1$ and $X_2$, respectively. Although mass action kinetics is the most common kinetic form, others are commonly used in biochemistry, including Michaelis-Menten kinetics \cite{M-M} and Hill kinetics \cite{Hi}.

\begin{definition}
\label{mas}
Consider a chemical reaction network $(\mathcal{S},\mathcal{C},\mathcal{R})$. Let $\mathbf{x}(t) = (x_1(t),\ldots,x_m(t)) \in \mathbb{R}_{\geq 0}^m$ denote the vector of reactant concentrations at time $t \geq 0$, and $\mathcal{K} = \{ k(i,j) > 0 \; | \; C_i \to C_j \in \mathcal{R} \}$ denote a set of rate constants. Then the \textbf{mass action system} $(\mathcal{S},\mathcal{C},\mathcal{R},\mathcal{K})$ associated with the network $(\mathcal{S},\mathcal{C},\mathcal{R})$ and rate constant set $\mathcal{K}$ is given by
\begin{equation}
\label{de1}
\frac{d\mathbf{x}}{dt} = Y \cdot A(\mathcal{K}) \cdot \Psi(\mathbf{x}(t))
\end{equation}
where
\begin{enumerate}
\item
$Y \in \mathbb{Z}_{\geq 0}^{m \times n}$ is the \textbf{stoichiometric matrix} with entries $Y_{i,j} = y_{ji}$;
\item
$A(\mathcal{K}) \in \mathbb{R}^{n \times n}$ is the \textbf{Laplacian matrix} with entries
\[[A(\mathcal{K})]_{i,j} = \left\{ \begin{array}{ll} \displaystyle{-\sum_{l=1}^r k(i,l),} \; \; \; \; \; & \mbox{if } i = j \\ k(j,i), & \mbox{otherwise;} \end{array} \right.\]
\item
$\Psi(\mathbf{x}) \in \mathbb{R}_{\geq 0}^n$ is the vector with entries $\Psi_i(\mathbf{x}) = \displaystyle{\prod_{j=1}^m x_j^{y_{ij}}}$.
\end{enumerate}
\end{definition}

The form (\ref{de1}) emphasizes the connectivity structure of network. In particular, we have $[A(\mathcal{K})]_{ji} > 0$ for $i \not= j$ if and only if $C_i \to C_j \in \mathcal{R}$. That is, the distribution of zero and non-zero off-diagonal elements of $A(\mathcal{K})$ encode the structure of the network.

It can also be easily shown that (\ref{de1}) can be represented in the alternative, and somewhat more intuitive, form:
\begin{equation}
\label{de2}
\frac{d\mathbf{x}}{dt} = \sum_{C_i \to C_j \in \mathcal{R}} k(i,j) \; (y_j - y_i) \; \prod_{l=1}^m x_l^{y_{il}}.
\end{equation}
It follows immediately from (\ref{de2}) that $\mathbf{x}'(t) \in S$. Nonnegativity of solutions is also well-known so that $\mathbf{x}(t) \in \mathsf{C}_{\mathbf{x}_0}$ for all trajectories $\mathbf{x}(t)$ of (\ref{de1}), where $\mathsf{C}_{\mathbf{x}_0} = (\mathbf{x}_0 + S) \cap \mathbb{R}_{> 0}^m$ is the stoichiometric compatibility class associated with the initial condition $\mathbf{x}_0 \in \mathbb{R}_{> 0}^m$ \cite{V-H}.

\subsection{Dynamical Equivalence and Linear Conjugacy}
\label{dynamicalequivalence}

It is well-known that structurally distinct chemical reaction networks can generate the same mass action system (\ref{de1}) under the assumption of mass action kinetics \cite{H-J1,C-P,J-S2}. For example, consider the following edge-weighted networks where the weights correspond to the value of the rate constant:
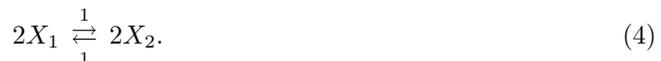
\begin{equation}
\label{network1}
2X_1 \stackrel{1}{\longrightarrow} 2X_2 \stackrel{2}{\longrightarrow} X_1 + X_2
\end{equation}
and
\begin{equation}
\label{network2}
2X_1 \; \mathop{\stackrel{1}{\rightleftarrows}}_{1} \; 2X_2.
\end{equation}
It can easily be seen that both of these networks generate the system of differential equations $\dot{x}_1 = -\dot{x}_2 = -2x_1^2+2x_2^2$ under the assumption of mass action kinetics (\ref{de1}). The networks are therefore said to be dynamically equivalent. Note, however, that the connectivity properties of (\ref{network1}) and (\ref{network2}) are different. In particular, (\ref{network2}) is weakly reversible while (\ref{network1}) is not.

The notion of linear conjugacy of mass action systems was introduced in \cite{J-S2}. Two systems $(\mathcal{S},\mathcal{C},\mathcal{R},\mathcal{K})$ and $(\mathcal{S}^{\star},\mathcal{C}^{\star},\mathcal{R}^{\star},\mathcal{K}^{\star})$ are said to be linearly conjugate to one another if the trajectories $\mathbf{x}(t)$ and $\mathbf{x}^{\star}(t)$ of (\ref{de1}) are related by $x_i(t) = c_ix^{\star}_i(t)$, $i=1, \ldots, n$, for some constants $c_i > 0$. For example, the networks
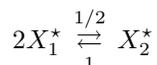
\[2X_1 \stackrel{1}{\longrightarrow} X_1 + X_2, \; \; \; X_2 \stackrel{1}{\longrightarrow} X_1\]
and
\[2X^{\star}_1 \; \mathop{\stackrel{1/2}{\rightleftarrows}}_{1} \; X^{\star}_2\]
have the non-dynamically equivalent mass action systems $\dot{x}_1 = -\dot{x}_2 = -x_1^2+x_2$ and $\dot{x}^{\star}_1 = -2\dot{x}^{\star}_2 = -(x^{\star}_1)^2+2x^{\star}_2$, respectively. These systems are linearly conjugate to one another under the transformation $x_1(t) = x^{\star}_1(t)$ and $x_2(t) = 2x^{\star}_2(t)$. Linearly conjugate systems share may be qualitative properties, including the number and stability of positive steady states, and the properties of persistence and boundedness (see Lemma 3.2 of \cite{J-S2}).

Since dynamically equivalent systems are a subset of linearly conjugate systems taking $c_i = 1$ for all $i=1, \ldots, m$, in the rest of this paper we will only refer to the study of linearly conjugate systems. We will do this with the understanding that the systems studied may, in fact, be dynamically equivalent.

\subsection{Deficiency Theory}
\label{deficiencytheory}

The following network parameter was introduced in \cite{F1} and \cite{H}, and has been studied significantly since \cite{Fe2,F2,Sh-F,BorosThesis}.

\begin{definition}
\label{deficiency}
The deficiency of a chemical reaction network $(\mathcal{S},\mathcal{C},\mathcal{R})$ is given by $\delta = n - \ell - s$ where $n = |\mathcal{C}|$, $\ell = |\mathcal{L}|$, and $s = \mbox{dim}(S)$.
\end{definition}

\noindent The deficiency is a nonnegative integer which may be determined based on study of the network topology alone. That is, it is independent of the rate constants and even the assumption of mass action kinetics (e.g. Michaelis-Menten, Hill kinetics). Further connections between the deficiency, different rate forms, and the matrices $Y$ and $A_k$ from Definition \ref{mas} are well-known but will not be summarized here \cite{F1,Gunawardena,Arceo2015}.

The following classical result was first presented in \cite{F1,H,H-J1}.

\begin{theorem}[Deficiency Zero Theorem]
\label{dzt}
Consider a chemical reaction network $(\mathcal{S},\mathcal{C},\mathcal{R})$ which satisfies the following:
\begin{enumerate}
\item
the network is weakly reversible; and
\item
the deficiency is zero (i.e. $\delta = 0$).
\end{enumerate}
Then, for all rate constant sets $\mathcal{K}$ and initial conditions $\mathbf{x}_0 \in \mathbb{R}_{> 0}^m$, the corresponding mass action system $(\mathcal{S},\mathcal{C},\mathcal{R},\mathcal{K})$ has the property that there exists a unique positive steady state $\mathbf{x}^* \in \mathsf{C}_{\mathbf{x}_0}$ and this state is locally asymptotically stable with respect to $\mathsf{C}_{\mathbf{x}_0}$.
\end{theorem} 

\noindent This result is surprising since, as just noted, the deficiency depends solely upon the network structure and not upon the choice of kinetics. The result, however, gives conclusions on the admissible dynamics to the corresponding mass action system (\ref{de1}); in fact, it gives very strong conclusions. It is also notable that the result holds independent of the rate constants and initial conditions; that is, it is robust to all of the system's parameter values.

In practice, many reaction networks arising from industrial chemistry and systems biology do not satisfy the assumptions of Theorem \ref{dzt}. The following result applies to many networks with a higher deficiency \cite{Fe2,F2}.

\begin{theorem}[Deficiency One Theorem]
\label{dot}
Consider a chemical reaction network $(\mathcal{S},\mathcal{C},\mathcal{R})$ with linkage classes $\mathcal{L} = \{ L_1, \ldots, L_{\ell} \}$. Let $\delta_{\theta}$ denote the deficiency of the subnetwork consisting only of the complexes and reactions in $L_{\theta}$ for $\theta=1, \ldots, \ell$. Suppose that:
\begin{enumerate}
\item
$\delta_{\theta} \leq 1$, for all $\theta=1,\ldots,\ell$,
\item
$\displaystyle{\sum_{\theta=1}^{\ell} \delta_{\theta} = \delta}$
\item
Every linkage class contains exactly one terminal strong linkage class (i.e. $t = \ell$).
\end{enumerate}
Then, if the mass action system $(\mathcal{S},\mathcal{C},\mathcal{R},\mathcal{K})$ admits a strictly positive steady state, every stoichiometric compatibility class contains exactly one steady state. Furthermore, if the network is weakly reversible, then $(\mathcal{S},\mathcal{C},\mathcal{R},\mathcal{K})$ admits a positive steady state for all choices of $\mathcal{K}$.
\end{theorem}

\noindent The conclusions of the Deficiency One Theorem are not as strong as the Deficiency Zero Theorem since it does not give any information about the stability of steady states; in fact, the steady states may be stable or unstable. It should be noted, however, that it is often very difficult to ascertain directly the uniqueness of steady states in compatibility classes and that the Deficiency One Theorem presents a parameter-free method for obtaining this property.

It is also worth noting that, despite the implication of the name, mechanisms satisfying the Deficiency One Theorem \emph{are not} required to have a deficiency of one. For example, it is permissible to have a network with a deficiency of two and two linkage classes with subnetwork deficiencies of one (i.e. $\delta_1 = 1$, $\delta_2 = 1$, and $\delta = \delta_1 + \delta_2 = 2$).

Further consequences of condition $2.$ of the Deficiency One Theorem were considered in \cite{BorosThesis}. The following result was shown.

\begin{theorem}[Corollary 3.5, \cite{BorosThesis}]
\label{boros}
Consider a chemical reaction network $(\mathcal{S},\mathcal{C},\mathcal{R})$ for which condition 2. of the Deficiency One Theorem holds. Then, if any stoichiometric compatibility class has a finite number of steady states, every stoichiometric compatibility class has the same finite number of steady states.
\end{theorem}

\noindent This result eliminates the possibility of bifurcations in the initial conditions resulting in different numbers of steady states. The result again depends only upon the structural information of the network, and not on the parameter values. It is often very challenging to ascertain this information through direct analysis of the differential equations (\ref{de1}).

\begin{example}
Consider the following chemical reaction network:
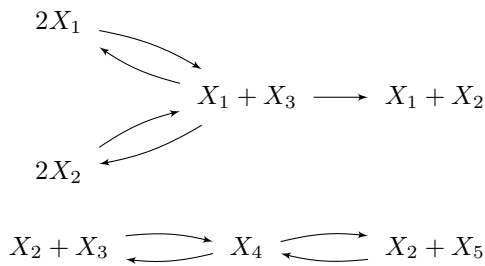
\begin{center}
\begin{tikzpicture}[auto, outer sep=3pt, node distance=2cm,>=latex']
\node (2X1) {$2X_1$};
\node [below of = 2X1, node distance = 2cm] (2X2) {$2X_2$};
\node [below of = 2X1, node distance = 1cm] (ghost) {};
\node [right of = ghost, node distance = 2.5cm] (X1X3) {$X_1+X_3$};
\node [right of = X1X3, node distance = 2.5cm] (X1X2) {$X_1+X_2$};
\node [below of = 2X2, node distance = 1cm] (X2X3) {$X_2+X_3$};
\node [right of = X2X3, node distance = 2.5cm] (X4) {$X_4$};
\node [right of = X4, node distance = 2.5cm] (X2X5) {$X_2+X_5$};
\path[->,bend left=10] (2X1) edge node {} (X1X3);
\path[->,bend left=10] (X1X3) edge node {} (2X1);
\path[->,bend left=10] (2X2) edge node {} (X1X3);
\path[->,bend left=10] (X1X3) edge node {} (2X2);
\path[->] (X1X3) edge node {} (X1X2);
\path[->,bend left=10] (X2X3) edge node {} (X4);
\path[->,bend left=10] (X4) edge node {} (X2X3);
\path[->,bend left=10] (X4) edge node {} (X2X5);
\path[->,bend left=10] (X2X5) edge node {} (X4);
\end{tikzpicture}
\end{center}
It can be quickly computed that that the stoichiometric space has dimension $s = \mbox{dim}(S) = 4$ so that the deficiency is $\delta = n - \ell - s = 7 - 2 - 4 = 1$. The deficiencies of the two linkage classes, enumerated in the order they appear above, are given by
\[\begin{split} \delta_1 & = n_1 - 1 - s_1 = (4) - 1 - (2) = 1 \\ \delta_2 & = n_2 - 1 - s_2 = (3) - 1 - (2) = 0.\end{split}\]
It follows that $\delta_1 \leq 1$, $\delta_2 \leq 1$, and $\delta = \delta_1 + \delta_2 =1$. The terminal strong linkage classes are $\{X_1 + X_2\}$ and $\{X_2 + X_3, X_4, X_2+X_5\}$ so that every linkage class only has a single terminal strong linkage class. The Deficiency One Theorem therefore applies so that, if the mass action system $(\mathcal{S},\mathcal{C},\mathcal{R},\mathcal{K})$ admits a positive steady state for some rate constant set $\mathcal{K}$, then every stoichiometric compatibility class has exactly one positive steady state.

\end{example}

\section{Main Results}
\label{computational}

In this section, we consider the question of whether, given a mass action system $(\mathcal{S},\mathcal{C},\mathcal{R},\mathcal{K})$, we can find a linearly conjugate system $(\mathcal{S}^{\star},\mathcal{C}^{\star},\mathcal{R}^{\star},\mathcal{K}^{\star})$ which satisfies the assumptions of the Deficiency One Theorem.

We will show that the answer is a definite yes. We present a MILP framework which checks whether conditions $1.$, $2.$ and $3.$ of the Deficiency One Theorem can be satisfied for a linearly conjugate system. The framework is an extension of recent work by the author and others on various problems within CRNT. The most directly applicable background paper is \cite{J-S6}, where the authors present a framework for determining whether a system is linearly conjugate to a system with an underlying network which satisfies the Deficiency Zero Theorem. We also take elements from the recent computational paper \cite{J2} which introduces a method for checking conditions on the connectivity of non-weakly reversible networks by relating them to a weakly reversible network.

We will at various points require the following background results from CRNT.

\begin{theorem}[Theorem 3.1 of \cite{G-H}; Proposition 4.1 of \cite{F3}]
\label{weaklyreversible}
Let $A(\mathcal{K})$ be the Laplacian matrix of a mass action system $(\mathcal{S},\mathcal{C},\mathcal{R},\mathcal{K})$ and let $\Lambda_{\theta}$, $\theta=1, \ldots, \ell,$ denote the support of the $\theta^{th}$ linkage class, $L_{\theta}$. Then the reaction graph corresponding to $A(\mathcal{K})$ is weakly reversible if and only if there is a basis of ker$(A(\mathcal{K}))$, $\left\{ \mathbf{w}^{(1)}, \ldots, \mathbf{w}^{(\ell)} \right\}$, such that, for $\theta=1, \ldots, \ell$,
\[\mathbf{w}^{(\theta)} = \left\{ \begin{array}{ll} w^{(\theta)}_j > 0, \hspace{0.3in} & j \in \Lambda_{\theta} \\ w^{(\theta)}_j = 0, & j \not\in \Lambda_{\theta}. \end{array} \right.\]
\end{theorem}

\begin{theorem}[Lemma 4.1, \cite{F1}]
\label{theorem1}
Let $\mathcal{L} = \{ L_1, \ldots, L_{\ell} \}$ denote the linkage classes of a chemical reaction network $(\mathcal{S},\mathcal{C},\mathcal{R})$. Then
\[S = \bigcup_{\theta=1}^{\ell} \mbox{span} \left\{ y_j - y_i \; | \; C_i, C_j \in L_{\theta}\right\}.\]
\end{theorem}

\begin{definition}
\label{kineticsubspace}
The \textbf{kinetic subspace} of a mass action system $(\mathcal{S},\mathcal{C},\mathcal{R},\mathcal{K})$ is the smallest subspace of $\mathbb{R}^m$ which contains $\mbox{im}(Y \cdot A(\mathcal{K}) \cdot \Psi(\mathbf{x}(t))$.
\end{definition}

\begin{theorem}[Unnumbered Theorem, \cite{H-F}]
\label{feinberghorn}
Consider a chemical reaction network $(\mathcal{S},\mathcal{C},\mathcal{R})$ for which every linkage class contains exactly one terminal strong linkage class (i.e. $t = \ell$). Then, regardless of the choice of $\mathcal{K}$, the dimension of the kinetic subspace of the mass action system $(\mathcal{S},\mathcal{C},\mathcal{R},\mathcal{K})$ is equal to $s = \mbox{dim}(S)$.
\end{theorem}

\begin{lemma}
\label{noproblem}
Let $(\mathcal{S},\mathcal{C},\mathcal{R})$ denote a chemical reaction network. Suppose $C'$ is a complex not contained in $\mathcal{C}$ and $R'$ is the self-loop $C' \to C'$. Then $(\mathcal{S},\mathcal{C},\mathcal{R})$ satisfies the Deficiency One Theorem if and only if $(\mathcal{S},\mathcal{C} \cup C',\mathcal{R}\cup R')$ satisfies the Deficiency One Theorem.
\end{lemma}

\begin{proof}
Suppose $(\mathcal{S},\mathcal{C},\mathcal{R})$ satisfies the Deficiency One Theorem. The network $(\mathcal{S},\mathcal{C}\cup C',\mathcal{R} \cup R')$ differs from $(\mathcal{S},\mathcal{C},\mathcal{R})$ only in the isolated complex $C'$, which is its own linkage class, and the self-reaction $R'$, which does not affect the stoichiometric subspace. It follows that $n' = n+1$, $\ell' = \ell +1$, and $s'=s$ where primes are used to denote the network qualitities associated with $(\mathcal{S},\mathcal{C}\cup C',\mathcal{R} \cup R')$. We give the isolated complex linkage class the index $\ell' = \ell +1$. It follows that the linkage classes deficiencies coincide for $\theta = 1, \ldots, \ell$, and that the isolate linkage class gives $\delta_{\ell'} = n_{\ell'} - 1 - s_{\ell'} = (1) - 1 - (0) = 0 \leq 1$. We furthermore have that $\delta = n - \ell - s = (n+1) - (\ell + 1) - s = n' - \ell' - s' = \delta'$ so that $\delta = \sum_{\theta = 1}^{\ell} \delta_{\theta}= \sum_{\theta = 1}^{\ell'}\delta'_{\theta} = \delta'$ where $\delta_{\theta} = \delta'_{\theta}$ for $\theta=1, \ldots, \ell$, and $\delta'_{\ell'} = \delta'_{\ell+1} = 0$. The isolated complex linkage class clearly also  contains a single terminal strong linkage class, corresponding to the complex itself. It follows that $(\mathcal{S},\mathcal{C}\cup C',\mathcal{R} \cup R')$ satisfies the Deficiency One Theorem. The argument holds in reverse, so that the result is slown.
\end{proof}

\noindent Notice that Lemma \ref{noproblem} may be extended to networks with an arbitrary number of isolated complexes added to or removed from a given network.

\subsection{Mixed-Integer Linear Programming Framework}
\label{optimization}

A MILP problem may be written in the form
\begin{equation}
\label{milp}
\begin{split}
\mbox{minimize} \; \; &\mathbf{c} \cdot \mathbf{x} \\
\mbox{subject to} \; \; & \left\{ \begin{array}{l} A_1 \cdot \mathbf{x} = \mathbf{b}_1 \\ 
A_2 \cdot \mathbf{x} \leq \mathbf{b}_2 \\
x_i \mbox{ is an integer for } i \in I, I \subseteq \{ 1, \ldots, n\} \end{array} \right.
\end{split}
\end{equation}
where $\mathbf{x} \in \mathbb{R}^n$ is a vector of unknown decision variables and $\mathbf{c} \in \mathbb{R}^n$, $\mathbf{b}_1 \in \mathbb{R}^{p_1}$, $\mathbf{b}_2 \in \mathbb{R}^{p_2}$, $A_1 \in \mathbb{R}^{p_1 \times n}$, and $A_2 \in \mathbb{R}^{p_2 \times n}$ are vectors and matrices of known parameters values \cite{Sz2}.

If all of the decision variables in the problem are real-valued then the problem (\ref{milp}) can be solved in polynomial time. If any of the variables are required to be integer-valued, however, (\ref{milp}) becomes NP-hard. The development of algorithms for efficiently solving MILP problems is a major area of current work which we do not summarize here. For the work in this paper, we utilize the non-commercial software packages GNU Linear Program Kit (GLPK) \cite{Makhorin2010} and SCIP \cite{Achterberg2009}.

\subsection{Initialization of Program}
\label{initialization}

We now set up the objectives of our MILP problem.

Consider two mass action systems $(\mathcal{S},\mathcal{C},\mathcal{R},\mathcal{K})$ and $(\mathcal{S}^{\star},\mathcal{C}^{\star},\mathcal{R}^{\star},\mathcal{K}^{\star})$. We will refer to the former as the \emph{original system} and the latter as the \emph{target system}. We wish to determine a network structure and associated rate constants for the target network consistent with the requirements of linear conjugacy and either the Deficiency One Theorem or Theorem \ref{boros}. We will suppose that the two networks have common species and complex sets (i.e. $\mathcal{S} = \mathcal{S}^{\star}$ and $\mathcal{C} = \mathcal{C}^{\star}$), that $\mathcal{R}$ and $\mathcal{K}$ are a priori known, and that $\mathcal{R}^{\star}$ and $\mathcal{K}^{\star}$ are to be determined. Notice that the first condition implies that the two  systems have a common matrix $Y$ and vector $\Psi(\mathbf{x})$. We will assume a priori that the following quantities are known:
\begin{itemize}
\item[--]
A complex matrix $Y \in \mathbb{R}_{\geq 0}^{m \times n}$ which is common to both the original and target systems.
\item[--]
Either the reaction structure $\mathcal{R}$ and rate constant set $\mathcal{K}$ of the original system or the numerical values of $M = Y \cdot A(\mathcal{K})$ corresponding to the coefficient map of $\Psi(\mathbf{x})$ in the mass action system (\ref{de1}). Notice that $M$ may be determined from $Y$ and $\mathcal{K}$ if they are known.
\item[--]
The dimension of the kinetic space of the original system, which we will denote by $s$. This value is easily computed as the rank of the matrix $M$. The equivalence of the kinetic subspace of the original network and stoichiometric subspace of the target system follows from condition $3.$ of the Deficiency One Theorem and Theorem \ref{feinberghorn}.
\item[--]
A small parameter $\epsilon > 0$.
\item[--]
A set of random variables $\delta[i,j]$ chosen uniformly from the range $[\sqrt{\epsilon},1/\sqrt{\epsilon}]$.
\end{itemize}

\noindent Since we do not know how many linkage classes the target system will contain, but must know this to calculate the deficiency by Definition \ref{deficiency}, we will use the upper bound $n-s$ (see \cite{J-S6}). We will also allow the networks to contain unused complexes and note that, if a complex does not appear in the target system, it may be treated as a single isolated linkage class with a self-reaction. By Lemma \ref{noproblem}, this does not affect our ability to apply the Deficiency One Theorem.

\subsection{Implementing Dynamical Equivalence and Linear Conjugacy}
\label{implementing1}

In this section, we introduce constraint sets which guarantee that the original and target systems are linearly conjugate to one another. A full description of the process can be found in \cite{J-S2} and \cite{J-S4}.

We introduce the following decision variables:
\begin{flalign}
\tag{\textbf{Dec1}}
\label{dec1}
&
\left\{ \; \; \; \begin{array}{ll} d[i] > \epsilon, & \; i=1,\ldots, m \\ b[i,j] \in [0, \frac{1}{\epsilon}], & \; i,j =1, \ldots, n, \; i \not= j. \end{array} \right.
&
\end{flalign}
where the $d[i]$ correspond to the reciprocals of the conjugacy constants $c_i$ (i.e. $d[i] = 1/c_i$) and the $b[i,j]$ correspond to scalings of the rate constants $k^{\star}(i,j)$ of the target network.

To see why we track $d[i]$ and $b[i,j]$, rather than $c_i$ and $k^{\star}(i,j)$, respectively, we briefly reproduce the arguments of \cite{J-S2}. We start by defining $T = \mbox{diag} \{ \mathbf{c} \}$. We can write the conjugacy transformation as $\mathbf{x} = T \cdot \mathbf{x}^{\star}$ so that the two networks are linearly conjugate for rate constant set $\mathcal{K}^{\star}$ if
\[Y \cdot A(\mathcal{K}) \cdot \Psi(\mathbf{x}) = T \cdot Y \cdot A(\mathcal{K}^{\star}) \cdot \Psi(\mathbf{x}^{\star}) = T \cdot A(\mathcal{K}^{\star}) \cdot \mbox{diag} \left\{ \Psi(\mathbf{c}) \right\} \cdot \Psi(\mathbf{x}).\]
We can simplify this expression by removing $\Psi(\mathbf{x})$ and making the substitutions $A(\mathcal{B}) = A(\mathcal{K}^{\star}) \cdot \mbox{diag} \left\{ \Psi(\mathbf{c}) \right\}$ and $M = Y \cdot A(\mathcal{K})$ (see Theorem 2 of \cite{J-S4}). After inverting $T$, we can see that the two networks are linearly conjugate if they satisfy the following linear constraint set:
\begin{flalign}
\tag{\textbf{LC}}
\label{linearlyconjugate}
&
\left\{ \; \; \; \begin{array}{l} \\[-0.1in] Y \cdot A(\mathcal{B}) = T^{-1} \cdot M. \\[0.05in] \end{array}\right.
&
\end{flalign}
$Y$ and $M$ are a priori known while $A(\mathcal{B})$ and $T^{-1}$ contain the unknown decision variables $b[i,j]$ and $d[i]$ (since $T^{-1} = \mbox{diag} \{ 1/\mathbf{c} \}$). Notice that the conjugacy constants can be obtained from $c_i = 1/d[i]$ and the rate constant set $\mathcal{K}^{\star}$ can then be recovered by the equation
\[A(\mathcal{K}^{\star}) = A(\mathcal{B}) \cdot \left[ \mbox{diag} \left\{ \Psi(\mathbf{c}) \right\} \right]^{-1}.\]
The distribution of positive and zero elements for $b[i,j]$ and $k^{\star}[i,j]$ coincide so that the two networks correspond to the same network structure $\mathcal{R}^{\star}$.


\subsection{Implementing Deficiency Conditions}
\label{deficiencyconditions}

In this section, we incorporate conditions $1.$ and $2.$ of the Deficiency One Theorem into a mixed integer linear programming framework. To accommodate Definition \ref{deficiency}, we need to be able to track the linkage classes and stoichiometric subspaces of both the original and target systems. We introduce the following decision variables:
\begin{flalign}
\tag{\textbf{Dec2}}
\label{dec2}
&
\left\{ \; \; \; \begin{array}{ll}
\Lambda[i,\theta] \in \{ 0, 1 \}, & \; i=1, \ldots, n, \; \theta = 1, \ldots n-s \\
\Gamma[i,j,\theta] \in \{ 0, 1 \}, & \; i,j=1,\ldots, n, \; i \not= j, \; \theta = 1, \ldots, n-s \\
S[i,j,\theta] \geq 0, & \; i,j=1,\ldots, n, \; i \not= j, \; \theta = 1, \ldots, n-s \\
S'[i,j,\theta] \in \{ 0, 1 \}, & \; i,j=1,\ldots, n, \; i \not= j, \; \theta = 1, \ldots, n-s \\
L[\theta] \in [0,1], & \; \theta = 1, \ldots, n-s .
\end{array} \right.
&
\end{flalign}

\noindent \emph{Counting linkage classes:} We will count the linkage classes of the target network using the techniques presented in \cite{J-S6}. To track the linkage classes of the target system, we desire the following logical equivalences:
\[\begin{split}
& \Lambda[i,\theta] = 1 \; \Longleftrightarrow \; C_i \in L^{\star}_{\theta}\\
& L[\theta] = 1 \; \Longleftrightarrow \; L^{\star}_{\theta} \not= \emptyset
\end{split}\]
where $L_{\theta}^{\star}, \theta =1, \ldots, n-s,$ are the linkage classes of the target system. We can accomplish this with the following set of constraints:
\begin{flalign}
\tag{\textbf{Linkage}}
\label{linkage}
&
\left\{ \; \; \; \begin{array}{ll}
\displaystyle{b[i,j] \leq \frac{1}{\epsilon}\cdot (\Lambda[i,\theta]-\Lambda[j,\theta]+1),} & \; i,j=1,\ldots, n, \; i \not= j, \; \theta = 1, \ldots, n-s \\
\displaystyle{\sum_{\theta = 1}^{n-s} \Lambda[i,\theta] = 1,}& \; i = 1, \ldots, n\\
\displaystyle{\sum_{i=1}^n \Lambda[i,\theta] - \epsilon \cdot L[\theta] \geq 0,} & \; \theta = 1, \ldots, n-s\\
\displaystyle{-\sum_{i=1}^n \Lambda[i,\theta] + \frac{1}{\epsilon} \cdot L[\theta] \geq 0,} & \; \theta = 1, \ldots, n-s\\
\displaystyle{\sum_{j=1}^i \Lambda[j,\theta] \geq \sum_{l = \theta+1}^{n-s} \Lambda[i,l],} & \; i =1, \ldots, n, \; \theta = 1, \ldots, n-s, \; \theta \leq i
\end{array} \right.
&
\end{flalign}
The first constraint guarantees that the scaled rate constant $b[i,j]$ is zero if $C_i$ and $C_j$ do not belong to the same linkage class $L_{\theta}$ in the target network. The second constraint guarantees that every complex is assigned to exactly one linkage class. The third and fourth constraints guarantee that $L[\theta]$ is zero if $L_{\theta}^{\star} = \emptyset$ and one if $L_{\theta}^{\star} \not= \emptyset$. That is, it counts the number of linkage classes in the network. The fifth constraint removes redundant permutations in the assignment of complexes to linkage classes and is necessary for computational efficiency (see \cite{J-S6} for full justification).\\

\noindent \emph{Computing stoichiometric subspace dimensions:} In order to determine the deficiency of each linkage class we need to compute $\delta_{\theta} = n_{\theta} - 1 - s_{\theta}$ where $n_{\theta}$ is the number of complexes in $L_{\theta}^{\star}$ and $s_{\theta}$ is the dimension of the stoichiometric subspace of the corresponding subnetwork.

We will compute $s_{\theta}$ by using the a priori known random variables $\delta[i,j]$ to construct a random vector in the span of the reaction vectors on the support on $L_{\theta}^{\star}$. Notice that this spans $S_{\theta}$ by Theorem \ref{theorem1}. With probability one, the dimension will correspond to the minimal number of vectors required to reach this random vector through a linear combination. We want the following logical equivalences:
\[\begin{split}
& \Gamma[i,j,\theta] = 1 \; \Longleftrightarrow \; C_i, C_j \in L_{\theta}^{\star} \\
& S'[i,j,\theta] = 1 \; \Longleftrightarrow \; y_j - y_i \mbox{ is a basis element of } S_{\theta} \\
& S[i,j,\theta] > 0 \; \Longleftrightarrow \; y_j -y_i \mbox{ is a basis element of } S_{\theta}.
\end{split}\]
We can accomplish this with the following constraints:
\begin{flalign}
\tag{\textbf{Stoic}}
\label{stoic}
&
\left\{ \; \; \; \begin{array}{ll}
S'[i,j,\theta] \leq \Gamma[i,j,\theta], & \; i,j=1,\ldots, n, \; i \not= j, \; \theta = 1, \ldots, n-s \\
\displaystyle{S[i,j,\theta] \leq \frac{1}{\epsilon}\cdot S'[i,j,\theta],} & \; i,j=1,\ldots, n, \; i \not= j, \; \theta = 1, \ldots, n-s \\
-S[i,j,\theta] \leq -\epsilon \cdot S'[i,j,\theta], & \; i,j=1,\ldots, n, \; i \not= j, \; \theta = 1, \ldots, n-s \\
\Gamma[i,j,\theta] \leq 1 + \epsilon \cdot (\Lambda[i,\theta]+\Lambda[j,l,\theta]-2),& \; i, j =1, \ldots, n, \; i \not= j, \; \theta = 1, \ldots, n-s\\
\Gamma[i,j,\theta] \geq \epsilon \cdot (\Lambda[i,\theta]+\Lambda[j,\theta]-1),& \; i, j =1, \ldots, n, \; i \not= j, \; \theta = 1, \ldots, n-s\\
\displaystyle{\mathop{\sum_{i,j=1}^{n}}_{i \not= j} S[i,j,\theta] \cdot (Y[k,j]-Y[k,i])} & \\= \displaystyle{\mathop{\sum_{i,j=1}^{n}}_{i \not= j} \Gamma[i,j,\theta] \cdot \delta[i,j] \cdot (Y[k,j]-Y[k,i]),}& \; \theta = 1, \ldots, n-s, \; k = 1, \ldots, n
\end{array} \right.
&
\end{flalign}
The first constraint restricts the basis vectors to those on the same linkage class. The second and third constraints guarantee that $S[i,j,\theta] \in [\epsilon,1/\epsilon]$ if $S'[i,j,\theta] = 0$ and $S[i,j,\theta] = 0$ if $S'[i,j,\theta]=0$. The fourth and fifth constraints guarantee that $\Gamma[i,j,\theta] = 1$ if and only if $C_i \in L_{\theta}^{\star}$ and $C_j \in L_{\theta}^{\star}$. The final constraint determines, for each linkage class, the number of vectors requires to reach a random vector in the corresponding subspace.\\

\noindent \emph{Conditions $1.$ and $2.$ of the Deficiency One Theorem:} We can now accommodate conditions $1.$ and $2.$ of the Deficiency One Theorem with the following constraint sets:
\begin{flalign}
\tag{\textbf{DOT}}
\label{dot}
&
 \left\{ \; \; \; \begin{array}{ll}
\displaystyle{-\mathop{\sum_{i,j=1}^n}_{i \not= j} S'[i,j,\theta] \leq 2  -\sum_{i=1}^n \Lambda[i,\theta],}& \; \theta = 1, \ldots, n-s\\
\displaystyle{\sum_{\theta =1}^{n-s} \mathop{\sum_{i,j=1}^{n}}_{i \not= j} S'[i,j,\theta] = s.}& \; \\
\end{array} \right.
&
\end{flalign}
The first constraint guarantees that $\delta_{\theta} \leq 1$ for all $\theta = 1, \ldots, n-s,$ while the second constraint guarantees that $\displaystyle{\sum_{\theta=1}^{n-s} s_{\theta} = s}$ so that $\displaystyle{\sum_{\theta=1}^{n-s} \delta_{\theta} = \sum_{\theta=1}^{n-s} \left( n_{\theta} - 1 - s_{\theta}\right) = n - \ell - s = \delta.}$ We can accommodate the simpler conditions of Theorem \ref{boros} by using the following constraint set as an alternative to (\ref{dot}):\\
\begin{flalign}
\tag{\textbf{Boros}}
\label{boros1}
&
 \left\{ \; \; \; \begin{array}{ll} \\[-0.1in]
\displaystyle{\sum_{\theta =1}^{n-s} \mathop{\sum_{i,j=1}^{n}}_{i \not= j} S'[i,j,\theta] = s.}\\[0.05in]
\end{array} \right.
&
\end{flalign}

\subsection{Implementing One Terminal Strong Linkage Class}
\label{implementing2}

In order to restrict the target network to contain only a single terminal strong linkage class in each linkage class, we introduce the following result.

\begin{lemma}
\label{lemma1}
Consider a chemical reaction network $(\mathcal{S},\mathcal{C},\mathcal{R})$ with linkage classes $\mathcal{L} = \{ L_1, \ldots, L_{\ell} \}$. Then every linkage class contains only a single terminal strong linkage class if and only if there is a set of complexes $\mathcal{C}' \subseteq \mathcal{C}$ and supplemental set of reactions $\mathcal{R}' \subseteq \mathcal{C} \times \mathcal{C}$ such that:
\begin{enumerate}
\item
$|\mathcal{C}' \cap L_{\theta}| \leq 1$ for every $\theta = 1, \ldots, \ell$;
\item
$(C_i,C_j) \in \mathcal{R}'$ implies that $C_i \in \mathcal{C}'$ and $C_i,C_j \in L_{\theta}$ for some $\theta \in \{ 1, \ldots, \ell \}$.
\item
The network $(\mathcal{S},\mathcal{C},\mathcal{R} \cup \mathcal{R}')$ is weakly reversible.
\end{enumerate}
\end{lemma}

\begin{proof}
($\Longrightarrow$) Suppose that the network has a single terminal strong linkage class in each linkage class. Let $\mathcal{C}'$ contain a single complex from each terminal strong linkage class and $\mathcal{R}'$ consist of all reactions from a complex in $\mathcal{C}'$ to a complex in the same linkage class. Then conditions $1.$, $2.$, and $3.$ are trivially satisfied.\\

\noindent ($\Longleftarrow$) Suppose that at least one linkage class has multiple terminal strong linkage classes. It follows by condition $1.$ and $2.$ that we may only introduce reactions which lead from a single terminal strong linkage class to any other complex. The remaining terminal strong linkage classes must remain terminal. It follows that we cannot satisfy condition $3.$ and we are done.
\end{proof}

This result says that the condition that every linkage class contains a single terminal strong linkage class can only be satisfied if supplemental reactions from at most one complex in each terminal strong linkage class can make the network weakly reversible. We will use this result in conjunction with Theorem \ref{weaklyreversible} to identify networks satisfying the Deficiency One Theorem.

To accommodate the conditions of Lemma \ref{lemma1}, and therefore condition $3.$ of the Deficiency One Theorem, we introduce the decision variables:
\begin{flalign}
\tag{\textbf{Dec3}}
\label{dec3}
&
\left\{ \; \; \; \begin{array}{ll}
w[i,j] \geq 0, & \; i,j =1,\ldots,n, \; i \not= j \\
w'[i,j] \geq 0, & \; i,j =1,\ldots,n, \; i \not= j \\
C[i,\theta] \in \{0, 1\}, & \; i=1,\ldots, n, \; \theta = 1, \ldots, n-s\\
C'[i] \in [0, 1], & \; i=1,\ldots, n.
\end{array} \right.
&
\end{flalign}\vspace{0.2in}

\noindent \emph{Condition $1.$ of Lemma \ref{lemma1}:} We wish to impose the following logical equivalences:
\[\begin{split}
& C[i,\theta] = 1 \; \Longrightarrow \; C_i \in \mathcal{C}' \mbox{ and } C_i \in L_{\theta} \\
& C[i] =1 \; \Longleftrightarrow \; C_i \in \mathcal{C}' \\
& |\mathcal{C}' \cap L_{\theta}|\leq 1 \mbox{ for every } \theta = 1, \ldots, \ell.
\end{split}\]
We can guarantee this with the following constraint set:
\begin{flalign}
\tag{\textbf{C'}}
\label{Cstar}
&
\left\{ \; \; \; \begin{array}{ll}
C[i,\theta] \leq \Lambda[i,\theta],& \; i =1, \ldots, n, \; \theta = 1, \ldots, n-s\\
\displaystyle{C'[i] = \sum_{\theta =1}^{n-s} C[i,l],}& \; i=1, \ldots, n\\
\displaystyle{\sum_{i=1}^n C[i,\theta] \leq 1,}& \; \theta = 1, \ldots, n-s.
\end{array} \right.
&
\end{flalign}
The first condition guarantees that the intermediate decision variable $C[i,\theta]$ attains the value one only if $C_i \in L_{\theta}$. The second condition determines whether $C_i \in \mathcal{C}'$. The third condition guarantees that each linkage class of the target network contains at most one such complex.\\ 

\noindent \emph{Conditions $2.$ and $3.$ of Lemma \ref{lemma1}:} We wish to impose the following:
\[\begin{split}
& w[i,j] > 0 \; \Longleftrightarrow \; b[i,j] > 0 \\
& w'[i,j] > 0 \; \Longleftrightarrow \; (C_i,C_j) \in \mathcal{R}'\\
& (\mathcal{S}^{\star},\mathcal{C}^{\star},\mathcal{R}^{\star}\cup\mathcal{R}') \mbox{ is weakly reversible.}
\end{split}\]
This can be accomplished with the following constraint sets:

\begin{flalign}
\tag{\textbf{WR}}
\label{wr}
&
\left\{ \begin{array}{ll}
w[i,j]-\epsilon \cdot b[i,j] \geq 0,& \; i,j = 1, \ldots, n, \; i \not= j\\
-w[i,j]+\frac{1}{\epsilon} \cdot b[i,j]  \geq 0,& \; i,j = 1, \ldots, n, \; i \not= j\\
w'[i,j] \leq C'[i],& \; i,j =1, \ldots, n, \; i \not= j.\\
w'[i,j] \leq \frac{1}{\epsilon} \cdot (\Lambda[i,\theta]-\Lambda[j,\theta]+1),& \; i,j = 1, \ldots, n, \; i \not= j, \; \theta = 1, \ldots, n-s\\
\displaystyle{\mathop{\sum_{j =1}^n}_{j \not= i} (w[i,j]+w'[i,j]) = \mathop{\sum_{j=1}^n}_{j \not= i} (w[j,i]+w'[j,i]),} & \; i = 1, \ldots, n
\end{array} \right.
&
\end{flalign}
The first two constraints guarantee that $w[i,j]$ and $b[i,j]$ have the same distribution of positive and zero elements. The third constraint guarantees only reactions from complexes in $\mathcal{C}'$ are allowed to be included in $\mathcal{R}'$, while the fourth constraint guarantees only reactions within linkage classes may be included in $\mathcal{R}'$. The final constraint guarantees the network $(\mathcal{S}^{\star},\mathcal{C}^{\star},\mathcal{R}^{\star} \cup \mathcal{R}')$ is weakly reversible according to Theorem \ref{weaklyreversible} (see \cite{J-S4}).\\


\noindent \emph{Objective function:} The existence of a target network satisfying the requirements of the Deficiency One Theorem is only dependent upon the feasible region of the MILP being non-empty. We still need, however, an objective function. It is often convenient to maximize the number of linkage classes, which minimizes the deficiency of the target network (see \cite{J-S6}). This can be accomplished with:
\begin{equation}
\tag{\textbf{Obj}}
\label{obj}
\mbox{minimize} \; \left\{ \; \; -\sum_{\theta =1}^{n-s} L[\theta] \right.\end{equation}
A target network satisfying the requirements of the Deficiency One Theorem can be found by optimizing (\ref{obj}) over the decision variables (\ref{dec1}), (\ref{dec2}), and (\ref{dec3}), and the constraint sets (\ref{linearlyconjugate}), (\ref{linkage}), (\ref{stoic}), (\ref{dot}), (\ref{Cstar}), and (\ref{wr}). Theorem \ref{boros} may be checked instead of the Deficiency One Theorem by replacing (\ref{dot}) with (\ref{boros1}).

\section{Examples}
\label{examples}

In this example, we apply the algorithm outlined in Section \ref{computational} to example mass action systems. All computations were performed on the author's professional use HP Spectre 360 laptop (Intel Core i7-5500U CPU @ 2.40 GHz, 8.00 GB RAM). The optimization programs used were GLPK \cite{Makhorin2010} and SCIP \cite{Achterberg2009}. \\

\begin{example}
\label{example1}
Consider the following chemical reaction network:
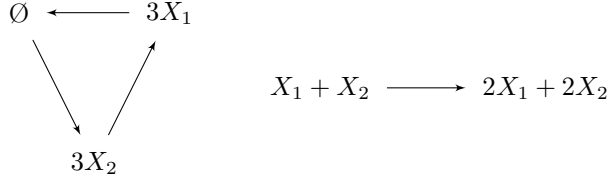
\begin{center}
\begin{tikzpicture}[auto, outer sep=3pt, node distance=2cm,>=latex']
\node (0) {$\O$};
\node [right of = 0, node distance = 2cm] (3X1) {$3X_1$};
\node [right of = 0, node distance = 1cm] (ghost1) {};
\node [below of = ghost1, node distance = 2cm] (3X2) {$3X_2$};
\node [above of = 3X2, node distance = 1 cm] (ghost2) {};
\node [right of = ghost2, node distance = 3cm] (X1X2) {$X_1+X_2$};
\node [right of = X1X2, node distance = 3cm] (2X12X2) {$2X_1+2X_2$};
\path[->] (0) edge node {} (3X2);
\path[->] (3X2) edge node {} (3X1);
\path[->] (3X1) edge node {} (0);
\path[->] (X1X2) edge node {} (2X12X2);
\end{tikzpicture}
\end{center}
where $\O$ corresponds to the zero complex, which has all zero stoichiometric coefficients. The zero complex is commonly used in CRNT to denote inflows and outflows of species in the system.

This network is not directly amenable to the Dezficiency One Theorem because we have $\delta = 1$ for the entire network but $\delta_1 = \delta_2 = 0$ for the two linkage class subnetworks so that $\delta \not= \delta_1 + \delta_2$. We will run the algorithm contained in Section \ref{computational} twice, both times taking all rate constants $k(i,j) = 1$ and $\epsilon = 0.1$.\\

\noindent Dynamical equivalency (i.e. $c_i = 1$ for $i=1, 2$): The algorithm produces the following network:
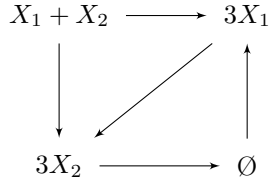
\begin{center}
\begin{tikzpicture}[auto, outer sep=3pt, node distance=2cm,>=latex']
\node (X1X2) {$X_1+X_2$};
\node [right of = X1X2, node distance = 2.5cm] (3X1) {$3X_1$};
\node [below of = 3X1, node distance = 2cm] (0) {$\O$};
\node [left of = 0, node distance = 2.5cm] (3X2) {$3X_2$};
\path[->] (0) edge node {} (3X1);
\path[->] (3X1) edge node {} (3X2);
\path[->] (3X2) edge node {} (0);
\path[->] (X1X2) edge node {} (3X1);
\path[->] (X1X2) edge node {} (3X2);
\end{tikzpicture}
\end{center}
with all rate constants $k^{\star}(i,j) = 1$. We have excluded the isolated complex $2X_1+2X_2$ and associated self-reaction since it is not relevant to our study by Lemma \ref{noproblem}. It can be easily checked that shown linkage classes has a single terminal strong linkage class consisting of $\{ \O, 3X_1, 3X_2 \}$, and that the deficiency is $\delta = 1$. Since the network consists of only the single linkage class, it follows that the system satisfies the Deficiency One Theorem.\\

\noindent Non-trivial linear conjugacy (i.e. $c_i$ varying): The algorithm produces the following network:
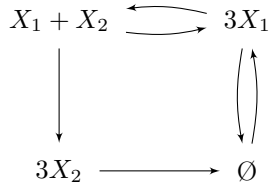
\begin{center}
\begin{tikzpicture}[auto, outer sep=3pt, node distance=2cm,>=latex']
\node (X1X2) {$X_1+X_2$};
\node [right of = X1X2, node distance = 2.5cm] (3X1) {$3X_1$};
\node [below of = 3X1, node distance = 2cm] (0) {$\O$};
\node [left of = 0, node distance = 2.5cm] (3X2) {$3X_2$};
\path[->, bend right = 10] (0) edge node {} (3X1);
\path[->, bend right = 10] (3X1) edge node {} (0);
\path[->] (3X2) edge node {} (0);
\path[->] (X1X2) edge node {} (3X2);
\path[->, bend right = 10] (X1X2) edge node {} (3X1);
\path[->, bend right = 10] (3X1) edge node {} (X1X2);
\end{tikzpicture}
\end{center}
with the values $c_1 = 1$, $c_2 = 4.66666$. To state the rate constants, we set $C_1 = \O$, $C_2 = 3X_1$, $C_3 = 3X_2$, and $C_4 = X_1 + X_2$. The program gives the rate constants $k^{\star}(1,2) = 1$, $k^{\star}(2,1) = 0.571429$, $k^{\star}(2,4) = 0.642857$, $k^{\star}(3,1) = 0.002109$, $k^{\star}(4,2) = 0.158163$, $k^{\star}(4,3) = 0.102041$, and the rest $k^{\star}(i,j) = 0$.

We again have that the network satisfies the Deficiency One Theorem. It is worth noting that this network is weakly reversible while insisting on dynamical equivalence yields a network which was not weakly reversible. This weak reversibility allows us to conclude that there does, in fact, exist a positive steady state, and so this steady state must be unique in the whole state space $\mathbb{R}_{> 0}^2$ since $S = \mathbb{R}^2$.
\end{example}

\begin{example}
\label{example2}
Consider the following mass action system:
\begin{equation}
\label{example2system}
\begin{split}
\dot{x}_1 & = 2x_2^3 - x_1^2 - x_1x_2x_3\\
\dot{x}_2 & = 1 - 3x_2^3+3x_1x_2x_3\\
\dot{x}_3 & = x_1x_2-x_1x_2x_3.
\end{split}
\end{equation}
This system can be converted into a chemical reaction network by the algorithm presented in \cite{H-T} and adapted to optimizing reaction network structures in \cite{Sz-H}. We will not reproduce the algorithm here. It yields the following chemical reaction network:
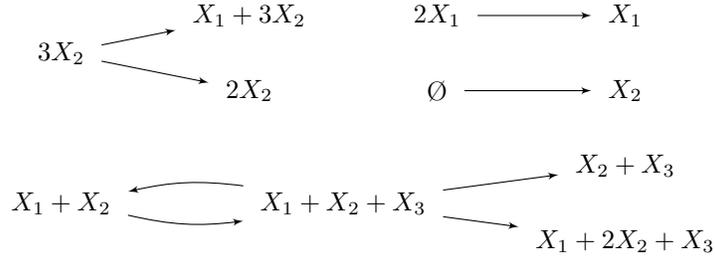
\begin{center}
\begin{tikzpicture}[auto, outer sep=3pt, node distance=2cm,>=latex']
\node (3X2) {$3X_2$};
\node [right of = 3X2, node distance = 2.5cm] (ghost) {};
\node [above of = ghost, node distance = 0.5cm] (X13X2) {$X_1+3X_2$};
\node [below of = ghost, node distance = 0.5cm] (2X2) {$2X_2$};
\node [right of = X13X2, node distance = 2.5cm] (2X1) {$2X_1$};
\node [right of = 2X1, node distance = 2.5cm] (X1) {$X_1$};
\node [right of = 2X2, node distance = 2.5cm] (0) {$\O$};
\node [right of = 0, node distance = 2.5cm] (X2) {$X_2$};
\node [below of = 3X2, node distance = 2cm] (X1X2) {$X_1+X_2$};
\node [right of = X1X2, node distance = 3.75cm] (X1X2X3) {$X_1 + X_2 + X_3$};
\node [right of = X1X2X3, node distance = 3.75cm] (ghost1) {};
\node [above of = ghost1, node distance = 0.5cm] (X2X3) {$X_2+X_3$};
\node [below of = ghost1, node distance = 0.5cm] (X12X2X3) {$X_1 + 2X_2 + X_3$};
\path[->] (3X2) edge node {} (X13X2);
\path[->] (3X2) edge node {} (2X2);
\path[->] (2X1) edge node {} (X1);
\path[->] (0) edge node {} (X2);
\path[->, bend right = 10] (X1X2) edge node {} (X1X2X3);
\path[->, bend right = 10] (X1X2X3) edge node {} (X1X2);
\path[->] (X1X2X3) edge node {} (X2X3);
\path[->] (X1X2X3) edge node {} (X12X2X3);
\end{tikzpicture}
\end{center}
This network is not weakly reversible, and has two linkage classes which have multiple terminal strong linkage classes. The overall deficiency is $\delta = 4$ while the deficiency of each linkage class is $\delta_{\theta} = 0$ for $\theta = 1, \ldots, 4$. It follows that neither the Deficiency One Theorem nor Theorem \ref{boros} apply to this network.

We now apply the algorithm provided in Section \ref{computational} to see if there is a linearly conjugate system satisfying either of these theorems. We take $\epsilon = 0.1$. The gives the following network:
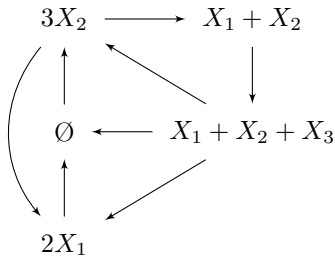
\begin{center}
\begin{tikzpicture}[auto, outer sep=3pt, node distance=2cm,>=latex']
\node (3X2) {$3X_2$};
\node [below of = 3X2, node distance = 1.5cm] (0) {$\O$};
\node [below of = 0, node distance = 1.5cm] (2X1) {$2X_1$};
\node [right of = 3X2, node distance = 2.5cm] (X1X2) {$X_1+X_2$};
\node [right of = 0, node distance = 2.5cm] (X1X2X3) {$X_1+X_2+X_3$};
\path[->, bend right = 40] (3X2) edge node {} (2X1);
\path[->] (0) edge node {} (3X2);
\path[->] (2X1) edge node {} (0);
\path[->] (3X2) edge node {} (X1X2);
\path[->] (X1X2) edge node {} (X1X2X3);
\path[->] (X1X2X3) edge node {} (3X2);
\path[->] (X1X2X3) edge node {} (0);
\path[->] (X1X2X3) edge node {} (2X1);
\end{tikzpicture}
\end{center}
where $c_1 = 1$, $c_2 = 0.9941642558$, and $c_3 = 0.2840295845$, and the remainder of the complexes are only involved in self-reactions (omitted). We index the complexes as $C_1 = \O$, $C_2 = 3X_2$, $C_3 = 2X_1$, $C_4 = X_1 +X_2$, and $C_5 = X_1 + X_2 + X_3$. The algorithm gives the rate constants $k^{\star}(1,2) = 0.335289$, $k^{\star}(2,3) = 0.999798$, $k^{\star}(2,4) = 0.358313$, $k^{\star}(3,1) = 0.5$, $k^{\star}(4,5) = 3.541427$, $k^{\star}(5,1) = 0.286590$, $k^{\star}(5,2) = 7.718363$, and $k^{\star}(5,3) = 4.463544$.

This network is weakly reversible and has a single linkage class. The deficiency of the network is $\delta = 1$ so that the Deficiency One Theorem applies. Since the stoichiometric subspace is $S = \mathbb{R}_{> 0}^3$, it follows that there is a unique positive steady state for the system (\ref{example2system}) in $\mathbb{R}_{>0}^3$. It is worth noting here that restricting to dynamical equivalence, and not full linear conjugacy, does not produce a network satisfying either the Deficiency Zero Theorem or Theorem \ref{boros}.
\end{example}

\section{Conclusions}
\label{conclusion}

In this paper, we have presented a MILP framework for determining whether a given mass action system is linearly conjugate to a system which satisfies the well-known Deficiency One Theorem or the recent generalization stated here as Theorem \ref{boros}. In particular, we have outlined constraint sets capable of imposing the following critical assumption: (a) that the sum of the deficiencies of each linkage classes is bounded by one and sums to the overall network deficiency; and (b) that each linkage class contains a single terminal strong linkage class. We also presented examples of systems for which the Deficiency One Theorem could be applied only after the algorithm was utilized.

This paper raises some interesting avenues for future work:
\begin{enumerate}
\item
\emph{Determination of optimal complex set:} Current MILP algorithms for determining optimal network structures within CRNT require that the complex set $\mathcal{C}$, and therefore $Y$, be specified prior to application of the algorithm. An initial complex set which is too small, however, runs the risk of not finding an admissible network, while a complex set which is too large increases computationally inefficiency. The network determination algorithm presented in \cite{H-T}, for instance, typically produces complex sets which are unnecessarily large (e.g. only $5$ of the $11$ complexes in Example \ref{example2} were needed in the linearly conjugate system). Determining methods of complex selection which bound the size of $Y$ is therefore a primary concern for future research.
\item
\emph{Parameter-free approach:} The MILP algorithm outlined in this paper depends upon the rate constants of the original network being specified. It is often beneficial, however, to leave this set unspecified; that is, to search over all possible mass action systems associated with a given network structure for a linearly conjugate system satisfying the Deficiency One Theorem. A parameter-free approach was introduced in \cite{J-S5} but is only known to be linear for dynamically equivalent relationships and not fully linearly conjugate ones. We saw in Example \ref{example2}, however, that the full application of linearly conjugate may be required to apply the Deficiency One Theorem. Extending the underlying theory to incorporate parameter-free approaches will therefore be the focus of future work.
\end{enumerate}

\noindent \textbf{Acknowledgments:} The author gratefully thanks San Jos\'{e} State University for its financial and logistic support, and Lake Tahoe for the serenity which lead to the conception of this project.


\end{document}